\documentclass[12pt]{elsarticle}
\usepackage{float}
\usepackage[utf8]{inputenc}
\usepackage{upgreek}
\usepackage{amssymb}
\usepackage{graphicx}
\usepackage{epstopdf}
\usepackage{latexsym}
\usepackage{amsmath}
\usepackage{amssymb}
\usepackage{latexsym}
\usepackage{amsmath}
\usepackage{amsfonts}
\usepackage{amsfonts}
\usepackage{array}
\usepackage{microtype}
\usepackage{amsthm}
\usepackage{hyperref}
\newtheorem{definition}{\bf Definition}[section]

\newtheorem{theorem}[definition]{\bf Theorem}

\newtheorem{remark}[definition]{\bf Remark}

\usepackage{hyperref}
\usepackage{xcolor}

\begin{document}

\begin{frontmatter}


\title{On set-convergence statistically modulated}


%\author{Fernando Le\'on-Saavedra}
%\address{ Department of Mathematics
%University of C\'adiz Avda. de la Universidad s/n %11402-Jerez de
%la Frontera. Spain.} \ead{fernando.leon@uca.es}

\author{Mar\'{\i}a Pilar Romero de la Rosa}

\address{ Department of Mathematics
University of C\'adiz Avda. de la Universidad s/n 11403-Jerez de
la Frontera. Spain.}
\ead{pilar.romero@uca.es}

\begin{abstract}

We explore some convergence notions for set-convergence coming from modern summability methods. Specifically we will see the connections between Wijsman $f$-statistical convergence and Wijsman $f$-strong Cesàro convergence, when $f$ is a modulus function that induces a density in $\mathbb{N}$. 
\end{abstract}

\begin{keyword}
Set convergence, Wijsman convergence, statistical convergence
\\ MSC \sep40H05  \sep   	40A35
\end{keyword}

\end{frontmatter}

\section{Introduction}

Set-convergences show up naturally in Mathematics. A century ago, it appear for the first time with the work of Painlevé (1902), that was supported and disseminate by Hausdorff and Kuratowski in their 1927 and 1933 books, respectively (see \cite[Chapter 4]{rocka} for historical details).
In the twenty century, the study of set convergence was centered on topological aspects by considering a set as a point in a topological space. These studies were collected on Beer's monograph (see \cite{beer}).
Currently, an increasing number on applied variational problems and  the constant queries coming from the modern data science, have renewed the interest in set-convergence from an applied point of view.

Imagine that we wish to know if a sequence of closed data sets $(A_n)_n$, obtaining experimentally, converges to a given closed subset $A$. Assume that  every subset is contained in a metric space $(X,d)$, and let us denote by $CL(X)$ the family of all  closed subsets of $X$. A way to see convergence of sets is to consider for any $x\in X$ the distances $(d(x,A_n))_n$ and to check if they converges to $d(x,A)$. In this case the sequence $A_n$
is said to be {\it Wijsman convergent} to $A$. However this is a strong condition, because the data sets were obtained experimentally and probably many data were collected erroneously. One way to relax this situation is to consider weak convergence methods.

For instance,  a sequence $(x_n)_n$, in a normed space, is said to be statistically convergent  to $L$ if for any $\varepsilon>0$ the subset
$A_\varepsilon=\{k\leq n\,:\,\|x_n-L\|>\varepsilon\}$ has zero density on $\mathbb{N}$. And with this weaker convergence in hand, Nuray and Rhoades \cite{rhoades} relaxed the concept of Wijsman convergence by saying that 
 a sequence of closed subsets $(A_k)_k$ is {\it Wijsman statistically convergent} to $A$ if for any $x\in X$ the distances $(d(x,A_n))_n$ converges statistically to $d(x,A)$. They also connect this concept with the concept of {\it Wijsman strong Cesàro convergence} which is a reminiscent of the classical strong Cesàro convergence. That is, a sequence  $(A_k)_k\subset CL(X)$ converges Wijsman strong Cesàro to $A\in CL(X)$ provided for any $x\in X$ the sequence of real numbers $(d(x,A_n))_n$ converges strong Cesàro to $L=d(x,A)$, that is:
 $$
\lim_{n\to\infty}\frac{1}{n} \sum_{k=1}^n |d(x,A_k)-d(x,A)|=0
 $$
However, sometimes statistical convergence is a drastic way
to filter data. And what is needed is to filter the data in a more accurate way. We can  obtain a fine filter data 
by means of a modulus function. Indeed, a modulus function will be a function $f\,:\, \mathbb{R}^+\to \mathbb{R}^+$ satisfying 
\begin{enumerate}
    \item $f(x)=0 $ if and only if $x=0$.
    \item $f(x+y)\leq f(x)+f(y)$ for every $x,y\in \mathbb{R}^+$.
    \item $f$ is increasing.
    \item $f$ is continuous from the right at $0$.
\end{enumerate}
Attempts to obtain convergence methods
by means of modulus functions date back to the works of Maddox \cite{maddox} in the nineties. 
The study of convergences through modulo functions has a wide applicability, not only for the case of convergence of sets. For instance, it can be used by F. García-Pacheco and R. Kama for refinements of the concept of derivative. And it have been used by C. Belen adn Yildrim to see new points of view for convergence of power series (\cite{compa2,compa3,compa1}).
 
 We will suppose that the modulus functions are unbounded to avoid trivialities. Let us denote by $\#(A)$  the cardinal of a finite subset $A$.
More recently, these weak convergences are transferred to convergence of sets in \cite{wij1}.  For instance, a sequence $(A_n)n\subset CL(X)$ is said to be Wijsman $f$-statistically convergent to $A$ if for any $x\in X$ the sequence of distances
$(d(x,A_n))_n$ converges $f$-statistically to $d(x,A)$ that is, for any $\varepsilon>0$
$$
\lim_{n\to \infty}\frac{1}{f(n)}f\left(\#\{k\leq n\,:\,|d(x,A_n)-d(x,A)|>\varepsilon \}\right)=0.
$$

Moreover, in \cite{wij1} there is a great  effort initiated by  V. K. Bhardwaj and S. Dhawan, to relate the $f$-statistical Wijsman convergence with the strong Cesàro convergence modulated by a modulus function, introduced by Maddox \cite{maddox}. However, as seen in previous studies \cite{jia1,jia2,racsam}, in order to obtain such a structure  between both convergences, it is necessary to introduce a smooth change in the notion of strong $f$-Cesàro convergence introduced by Maddox.

If it is used lacunary sequences in the framework of the classic concepts of statistical convergence and strong Cesàro convergence, it is also possible to open or close the degree of precision in filtering data sets. Freeman and Sember, in their 1978 visionary paper and inspired by  the classical studies on Cesàro convergence by Hardy-Littlewood and Fekete, studied the concept of strong lacunary convergence in \cite{sember}.
These studies were refined by Fridy and Orhan \cite{b5} by relating them to a new concept: {\it lacunary statistical convergence}, where it underlies new ways of measuring the density of subsets in $\mathbb{N}$.

Ulusu and Nuray \cite{ulusu} were the first to explore these convergences within the the framework of convergence of sets, studying the Wijsman lacunary statistical convergence  and the Wijsman lacunary strong  Cesàro  convergence. 
These rich Wijsman lacunary convergences were sharpened by   V. K. Bhardwaj, S. Dhawan \cite{wij2} using modulus functions.

In this paper we will complete the studies initiated in \cite{wij1} and \cite{wij2} by V. K. Bhardwaj, S. Dhawan and coworkers by fully characterizing the connection between Wijsman $f$-statistical convergence and Wijsman $f$-strong Cesàro convergence.
We explore the modulus functions that offer set-convergence methods which are less permissive with the errors obtained experimentally. Specifically this phenomenon happens when the modulus function $f$ is not compatible. On the contrary, when the modulus functions are compatible (see Definition \ref{compatible}) then the Wijsman $f$-statistically convergence is exactly the Wijsman statistically convergence introduced by Nuray and Rhoades \cite{rhoades}. We analyze also the existing structure between  the Wijsman convergence of sets when it is induced by lacunary $f$-statistical convergences.

The paper is structured as follows. In Section \ref{dos} we will explore the connections between Wijsman $f$-statistical convergence and Wijsman $f$-strong Cesàro convergence. In Section \ref{tres}  we fully
characterize the relationship between the Wijsman $f$-statistical convergence and the Wijsman $f$-strong Cesàro convergence in the lacunary setting.

\section{Compatible modulus functions versus Wijsman $f$-statistical convergence}
\label{dos}

In \cite{jia1} the concept of compatible modulus functions was introduced. Basically, for the compatible modulus functions it was obtained in \cite{jia1} that  $f$-statistical convergence and statistical convergence are equivalent.

\begin{definition}
\label{compatible}
Let us denote by $\varphi(\varepsilon)=\limsup_n \frac{f(n\varepsilon)}{f(n)}$. A modulus function  $f$ is said to be compatible if $\lim_{\varepsilon\to 0}\varphi(\varepsilon)=0$.
\end{definition}

\begin{remark}
Examples of compatible modulus functions are $f(x)=x^p+x^q$, $0<p,q\leq 1$, $f(x)=x^p+\log(x+1)$, $f(x)=x+\frac{x}{x+1}$. And $f(x)=\log(x+1)$, $f(x)=W(x)$ (where $W$ is the $W$-Lambert function restricted to $\mathbb{R}^+$, that is, the inverse of $xe^x$) are modulus functions which are not compatible.  By taking inverses, there is a plenty number of non-compatible modulus functions. For instance, we consider the inverse of the functions $x^pe^x$, which are generalized $W$-Lambert functions.

Let us show that $f(x)=x+\log(x+1)$ is compatible. 
$$
\lim_{n\to\infty}\frac{f(n\varepsilon')}{f(n)}=\lim_{n\to \infty}\frac{n\varepsilon'+\log(1+n\varepsilon')}{n+\log(n+1)}=\varepsilon'
$$
On the other hand  if $f(x)=\log(x+1)$, since
$$
\lim_{n\to\infty}\frac{\log(1+n\varepsilon')}{\log(1+n)}=1
$$
we obtain that $f(x)=\log(x+1)$ is not compatible.

\end{remark}
To simplify the language, let us denote by $A_n\overset{WS}{\longrightarrow}A$ when $(A_n)_n$ Wijsman statistically converges to $A$ and by 
$WS$ the set of all sequences in $CL(X)$ that are Wijsman statistically convergent. Analogously we denote by $A_n\overset{WS^f}{\longrightarrow}A$ if $(A_n)_n$ is Wijsman $f$-statistially convergent to $A$ and
by  $WS^f$ the set of all sequences in $CL(X)$ that are Wijsman $f$-statistically convergent.

From the point of view of data filtering, which is our main interest in this work, the following result shows that the interesting modulus functions are those that are not compatible. 

\begin{theorem}
\label{th1compatible}
Let $(X,d)$ be metric space. Assume that $f$ is a compatible modulus function, and $(A_k)_k\subset CL(X)$, $A\in CL(X)$. The following conditions are equivalents:
\begin{description}
    \item[a)] $(A_k)_k$ converges Wijsman $f$-statistically  to $A$.
    \item[b)] $(A_k)_k$ converges Wijsman statistically to $A$.
\end{description}
\end{theorem}
\begin{proof}
Condition  $a)$ implies $b)$ follows without the hypothesis on compatibility of $f$ and it was proved in \cite[Theorem 2.1]{wij1}.

Now assume that $A_n\overset{WS}{\longrightarrow}A$. For each $x\in X$
we have that $(d(x,A_n))_n$ converges statistically to $d(x,A)$. Since $f$ is compatible, according to Proposition 2.7 in \cite{jia1} we get that $(d(x,A_n))_n$ converges $f$-statistically to $d(x,A)$ for any $x\in X$, that is, 
$A_n\overset{WS^f}{\longrightarrow}A$ as desired.
\end{proof}
In brief, when $f$ is compatible then $WS=WS^f$. Moreover when this phenomenon happens then $f$ must be compatible.
\begin{theorem}
\label{converse1}
Assume that for some $f$ we have $WS=WS^f$ then $f$ is compatible.
\end{theorem}
\begin{proof}
Indeed, if $f$ is not compatible then there exists $c>0$ and a subsequence $\varepsilon_k$ such that $\varphi(\varepsilon_k)=\limsup_n\frac{f(n\varepsilon_k)}{f(n)}>c$. In fact since $\varphi(\cdot)$ is increasing, we have that $\varphi(\varepsilon)>c$ for any $\varepsilon>0$.

Let us fix a decreasing sequence $\varepsilon_k$ converging to $0$. For each $k$ we can select $m_k$ large enough such that $f(m_k\varepsilon_k)\geq c f(m_k)$.
We can select $m_{k+1}$ inductively satisfying
\begin{equation}
\label{desigualdad}
1-\varepsilon_{k+1}-\frac{1}{m_{k+1}}>\frac{(1-\varepsilon_k)m_k}{m_{k+1}}
\end{equation}
Now the idea is to construct a subset with prescribed density. Let us denote $\lfloor x\rfloor $ the integer part of $x\in\mathbb{R}$. Set $n_k=\lfloor m_k\varepsilon_k\rfloor+1$. Extracting a subsequence if it is necessary, we can assume that $n_1<n_2<\cdots$, $m_1<m_2<\cdots$. Let us define  $A_k=[m_{k+1}-(n_{k+1}-n_k)]\cap\mathbb{N}$. Condition (\ref{desigualdad}) is to guarantee that $A_k\subset [m_k,m_{k+1}]$. Let us denote $A=\bigcup_kA_k$.

Without loss of generality, we can suppose that the metric space is the complex plane with the euclidean metric. Let us consider the following sequence of subsets:
$$
B_k=\begin{cases}
\{1\} & k\in A\\
\{0\} & k\notin A.
\end{cases}
$$
Assume that $B=\{0\}$, and let us see that $B_k\overset{WS^f}{\nrightarrow} B$. Indeed, if we take $x_0=0$ then
$$
d(0,B_k)=\begin{cases}
1 & k\in A\\
0 & k\notin A
\end{cases}
$$
and $d(0,B)=0$ therefore
\begin{eqnarray*}
\frac{f(\#\{k\leq m_{k} \,/ |d(0,B_k)-d(0,B)|>1/2\})}{f(m_{k})}&=&\frac{f(n_k)}{f(m_k)} \\
&\geq& \frac{f(m_k\varepsilon_k)}{f(m_k)}\geq c.
\end{eqnarray*}

Now, let us show that $B_k\overset{WS}{\rightarrow}B$. Indeed, for any $m$, there exists $k$ such that $m_k<m\leq m_{k+1}$. Moreover, we can suppose without loss that $m\in A$, that is, $m_{k+1}-n_{k+1}+n_{k}\leq m$. Let us consider $x\in \mathbb{C}$ we can suppose without loss that $|1-x|\neq |x|$. Thus for any $\varepsilon>0$:
\begin{eqnarray*} 
\frac{\#\{l\leq m \,:\,|d(x,B_l)-d(x,B)|>\varepsilon\}}{m}&\leq& \frac{\#\{l\leq  m_{k} \,:\,|d(x,B_k)-d(x,B)|>\varepsilon\}}{m_k}\\&&+\frac{n_{k+1}-n_k}{m_{k+1}-n_{k+1}+n_{k}} 
\\
&\leq &\frac{n_k}{m_k}+\frac{1}{\frac{m_{k+1}}{n_{k+1}-n_k}-1}\to 0
\end{eqnarray*} 
as $k\to \infty$, which proves the desired result.
\end{proof}
Now, we turn our attention to the Wijsman $f$-strong Cesàro convergence.
Let us recall that a sequence $(x_n)_n$ is said to be $f$-strong Cesàro convergent to $L$ if
$$
\lim_{n\to\infty}\frac{f\left(\sum_{k=1}^n|x_k-L|\right)}{f(n)}=0.
$$
The above definition suggests the following notion.
\begin{definition}
Assume that $(X,d)$ is a metric space, $(A_k)_k\subset CL(X)$ and $A\in CL(X)$. We say that $(A_k)_k$ is  Wijsman $f$-strong Cesàro convergent 
to $A$ if for any $x\in X$ the sequence $(d(x,A_k))_k$ is $f$-strong 
Cesàro convergent to $d(x,A)$.
\end{definition}
A result similar to Theorem \ref{th1compatible} is true for the  Wijsman $f$-strong Cesàro convergence.
\begin{theorem}
\label{th2compatible}
Let $(X,d)$ be metric space. Assume that $f$ is a compatible modulus function, and $(A_k)_k\subset CL(X)$, $A\in CL(X)$. The following conditions are equivalents:
\begin{description}
    \item[a)] $(A_k)_k$ converges Wijsman $f$-strong Cesàro  to $A$.
    \item[b)] $(A_k)_k$ converges Wijsman strong Cesàro to $A$.
\end{description}
\end{theorem}
\begin{proof}
According to Proposition 2.3 in \cite{jia1}, since for any modulus function, a $f$-strong Cesàro convergent sequence 
is strong Cesàro convergent, we get that $a)$ implies $b)$ follows directly without using the hypothesis on compatibility on $f$.

Now, for each $x\in X$
we have that $(d(x,A_n))_n$ is strong Cesàro convergent   to $d(x,A)$. Since $f$ is compatible, according to Proposition 2.8 in \cite{jia1} we get that $(d(x,A_n))_n$ is $f$-strong Cesàro convergent to $d(x,A)$ for any $x\in X$, that is, 
$(A_n)_n$ is Wijsman $f$-strong Cesàro convergent to $A$ as we desired. 
\end{proof}
Let us denote by $WN$ the class of sequences of sets that are Wijsman strong Cesàro convergent and by $WN^f$ the class of sequences which are Wijsman $f$-strong Cesàro convergent. When $f$ is compatible then $WN=WN^f$. Now let us shown the converse result.

\begin{theorem}
\label{converse2}
Assume that for some $f$ we have $WN=WN^f$ then $f$ is compatible.
\end{theorem}
\begin{proof}
Indeed, let us consider the sequences $(\varepsilon_k)_k, (m_k)_k, (n_k)_k$ and $(A_k)_k$ provided by Theorem \ref{converse1}.
Let us denote $A=\bigcup_kA_k$.

Again without loss, we can suppose that the metric space is the complex plane with the euclidean metric. Let us consider the following sequence of subsets:
$$
B_k=\begin{cases}
\{1\} & k\in A\\
\{0\} & k\notin A.
\end{cases}
$$
Assume that $B=\{0\}$, and let us see that $B_k\overset{WN^f}{\nrightarrow} B$. Indeed, let us consider $x_0=0$, then
\begin{eqnarray*}
\frac{f\left(\sum_{k=1}^{m_k}|d(0,B_k)-d(0,B)|\right)}{f(m_k)} 
&=&
\frac{f(\#\{k\leq m_{k} \,/ |d(0,B_k)-d(0,B)|>1/2\})}{f(m_{k})}\\
&=&\frac{f(n_k)}{f(m_k)} \\
&\geq& \frac{f(m_k\varepsilon_k)}{f(m_k)}\geq c.
\end{eqnarray*}
However $B_k\overset{WN}{\rightarrow}B$. Indeed, since  for any $x\in X$
$\sum_{k=1}^{m}|d(x,B_k)-d(x,B)|=\#\{k\leq m \,/ |d(x,B_k)-d(x,B)|>\varepsilon\}$ and since by Theorem \ref{converse1} the sequence $(B_k)_k$ is Wijsman statistically convergent to $B$, we obtain that $(B_k)_k$ is Wijsman strong convergent to $B$ as we desired to prove.
\end{proof}
Now, to the end of this section, let us explore the relationship between the concept of Wijsman $f$-statistically convergence and Wijsman $f$-strong Cesàro convergence.

To establish this connection we need the extra hypothesis introduced in \cite{wij1}: {\it Wijsman boundedness} of a sequence of subsets. A sequence $(A_k)_k\subset CL(X)$ is said to be {\it Wijsman bounded} if for any $x\in X$ the sequence $(d(x,A_n))_n$ is bounded. Let us observe that by the triangular inequality a sequence $(A_k)_k\subset CL(X)$ is Wijsman bounded if there exists $x_0\in X$ such that $(d(x_0,A_n))_n$ is bounded. Let us observe also that Wijsman boundedness notion does not imply boundedness of the subsets $(A_k)$.

In fact, this boundedness condition  can be sharpened. Khan and Orhan in \cite{khan} discovered that the classical result that connect statistically convergence with strong Cesàro convergence can be established for uniform integrable sequences. 
Let $Y$ be a normed space. A sequence $(x_n)_n\subset Y$ is said to be {\it uniformly integrable} if
$$
\lim_{c\to \infty}\sup_n \frac{1}{n}\sum_{\substack{k=1 \\ \|x_k\|\geq c}}^n\|x_k\|=0.
$$
If a sequence $(x_n)_n$ is uniformly integrable then $(x_n-L)_n$ is also uniformly integrable for every $L\in Y$.

Our result pivots on this notion. Let us refine also the notion of Wijsman boundedness:
\begin{definition}
Let $(X,d)$ be a metric space. A sequence $(A_k)_k\in CL(X)$ is said to be
Wijsman uniform integrable if for any $x\in X$ the sequence $(d(x,A_n))_n$ is uniform integrable.
\end{definition}
Now, we are in position to state our result.
\begin{theorem}
Assume $(X,d)$ is a metric space and $(A_k)_k\subset CL(X)$, $A\in CL(X)$.
\begin{description}
\item[a)] If $(A_k)_k$ is Wijsman $f$-strong Cesàro convergent to $A$ then $(A_n)_n$ is Wijsman $f$-statistically convergent to $A$ and $(A_k)_k$ is Wijsman uniformly integrable.
\item[b)] Assume that $f$ is compatible. If $(A_k)_k$ is Wijsman $f$-statistically convergent to $A$ and $(A_k)_k$ is Wijsman uniformly bounded then $(A_k)_k$ is Wijsman $f$-strong Cesàro convergent to $A$.
\end{description}
\end{theorem}
\begin{proof}
To show a), let us observe that for any $x\in X$ the sequence $(d(x,A_n))_n$ is $f$-strong Cesàro convergent to $d(x,A)$. Therefore, according to Theorem 3.2 in \cite{jia1}, the sequence $(d(x,A_n))_n$ is $f$-statistically convergent to $d(x,A)$ and it is uniformly integrable. That is, $(A_k)_k$ is Wijsman $f$-statistically convergent to $A$ and Wijsman uniformly integrable.

On the other hand, if for any $x\in X$ the sequence $(d(x,A_n))_n$ is $f$-statistically convergent to $d(x,A)$ and it is uniformly integrable, since $f$ is compatible, according to Theorem 3.4 in \cite{jia1} we get that $(d(x,A_n))_n$ is $f$-strong Cesàro convergent to $d(x,A)$. Hence, $(A_k)_k$ is Wijsman $f$-strong Cesàro convergent to $A$ as we desired.
\end{proof}

There is a converse of the above result. Let us denote by $WI$
the set of all Wijsman uniformly integrable sequence of subsets. We have shown that for any modulus function $f$, $WN^f\subset WS^f\cap WI$, and 
if $f$ is compatible then
$WN^f=WS^f\cap WI$.
\begin{theorem}
Assume that $WN^f=WS^f\cap WI$ then $f$ is compatible.
\end{theorem}
\begin{proof}
If $f$ is not compatible. Thus, as in the proof of Theorem \ref{converse1} we can construct  sequences $(\varepsilon_k)_k$, $(m_k)_k$ such that
$f(m_k\varepsilon_k)\geq c f(m_k)$ for some $c>0$. Moreover, we can construct $(m_k)_k$ inductively, such that the sequence
$$
r_{k}= \frac{m_{k+1}\varepsilon_{k+1}-m_k\varepsilon_k}{m_{k+1}-m_k}
$$
is decreasing and converging to $0$.

Let us consider the following subsets:
$$
B_l=\begin{cases}
\{r_k\} & l\in A_k \\
\{0\} & k\notin A.
\end{cases}
$$
Since $(r_k)_k$ is decreasing, clearly $(B_k)$ is Wijsman bounded and therefore Wijsman integrable. Moreover since $(r_k)_k$ is decreasing to $0$, the sequence $(B_j)_j$ is Wijsman statistically convergent to $B=\{0\}$. However if we take $x_0=0$ then
$$
\frac{f\left(\sum_{l=1}^{m_k} |d(0,B_l)-d(0,B)| \right)}{f(m_k)}=\frac{f(m_k\varepsilon_k)}{f(m_k)}\geq 0
$$
which proves that $(B_k)_k$ is not Wijsman $f$-strong Cesàro convergent to $B=\{0\}$ as we desired to prove.
\end{proof}

\section{$\Theta$-compatible modulus functions versus Wijsman lacunary $f$-statistical convergence}
\label{tres}

By a lacunary sequence  $\theta=\{k_r\}_{r\geq 0}\subset \mathbb{N}$ with $k_0=0$, we will means   an increasing sequence of natural numbers such that $h_r=k_r-k_{r-1}\to \infty$ as $r\to \infty$. Assume that $f$ is a modulus function.
Denoting the intervals $I_r=(k_{r-1},k_r]$,
a sequence $(x_n)$ on a normed space $X$ is said to be $f$ strong lacunary convergent to $L$ provided (see \cite{sember})
$$
\lim_{r\to\infty}\frac{1}{f(h_r)}f\left(\sum_{k\in I_r}^{\infty}\|x_k-L\|\right)=0.
$$
and is said to be lacunary $f$-statistical convergent to $L$ (see \cite{b5}) if for any $\varepsilon >0$ the subset $A_\varepsilon=\{k\in\mathbb{N}\,:\,\|x_k-L\|>\varepsilon\}$ has 
($\theta, f)$-density equal $0$. That is,
$$
d_{\theta,f}(A_\varepsilon)=\lim_{r\to\infty}\frac{1}{f(h_r)}
f\left(\#\{k\in I_r\,:\, \|x_k-L\|>\varepsilon\}\right)=0.
$$

In \cite{wij2} V. K. Bhardwaj, S. Dhawan applied  these new convergence methods in the study of convergence of sets.
We say that a sequence $(A_k)_k\subset CL(X)$ is Wijsman lacunary $f$-statistical convergent to $A\in CL(X)$ if for any
$x\in X$ the sequence $(d(x,A_k))_k$  is lacunary $f$-statistical convergent to $d(x,A)$. Let us denote by $WS_\theta^f$ the set of all Wijsman $\theta$-lacunary $f$-statistical convergent sequences. 
Analogously we say that $(A_k)_k\subset CL(X)$ is Wijsman lacunary $f$-strong Cesàro convergent to $A\in CL(X)$ if for any
$x\in X$ the sequence $(d(x,A_k))_k$  is lacunary $f$-strong Cesàro convergent to $d(x,A)$. Let us denote by $WN_\theta^f$ the set of all Wijsman $\theta$-lacunary $f$-strong Cesàro convergent sequences.

Our results pivots on the following notion.
\begin{definition}
\label{thetacompatible}
Let us denote  $\varphi_{\theta}(\varepsilon)=\limsup_{t\to\infty}\frac{f(h_t\varepsilon)}{f(h_t)}$. We will say that  $f$ is $\theta$-compatible if  if $\lim_{\varepsilon\to 0}\varphi_\theta(\varepsilon)=0$. 
\end{definition}
Clearly, if $f$ is compatible then for any lacunary sequence $\theta$, $f$ is $\theta$-compatible.

\begin{theorem}
Let $(X,d)$ be a metric space and $(A_k)_k\subset CL(X)$, $A\in CL(X)$.  Assume that $\theta=(k_t)_t$ is a lacunary sequence.
\begin{description}
\item[a)] For any modulus function $f$, if $(A_n)_n$ is Wijsman lacunary $f$-statistical convergent to $A$ then $(A_n)_n$ also converges Wijsman lacunary statistical convergent to $A$.
\item[b)]  If $f$ is $\theta$-compatible and $(A_n)_n$ is Wijsman lacunary statistical convergent to $A$ then $(A_n)_n$ is Wijsman $f$-statistical convergent to $A$.
\end{description}
\end{theorem}
\begin{proof}
Indeed, part a) is exactly  Theorem 3.5 in \cite{wij2}.

Let us shown part b). If $(A_n)_n$ is Wijsman lacunary statistical convergent to $A$ then for any $x\in X$ the sequence $(d(x,A_n))_n$ is $\theta$-lacunary statistical convergent to $d(x,A)$. Hence, since $f$ is $\theta$-compatible, according to Theorem 3.2 in \cite{racsam} we get that $(d(x,A_n))_n$ is $\theta$-lacunary $f$-statistical convergent to  $d(x,A)$, therefore $(A_n)$ is Wijsman $\theta$-lacunary $f$-statistical convergent to $A$ as desired.
\end{proof}

\begin{theorem}
Let $(X,d)$ be a metric space and $(A_k)_k\subset CL(X)$, $A\in CL(X)$.  Assume that $\theta=(k_t)_t$ is a lacunary sequence.
\begin{description}
\item[a)] For any modulus function $f$, if $(A_n)_n$ is Wijsman $\theta$-lacunary $f$-strong Cesàro convergent to $A$ then $(A_n)_n$ also converges Wijsman $\theta$-lacunary statistical convergent to $A$.
\item[b)]  If $f$ is $\theta$-compatible and $(A_n)_n$ is Wijsman lacunary strong Cesàro convergent to $A$ then $(A_n)_n$ is Wijsman $\theta$-lacunary $f$-strong Cesàro convergent to $A$.
\end{description}
\end{theorem}
\begin{proof}
To show $a)$, assume that $(A_n)_n$ is Wijsman $\theta$-lacunary $f$-strong Cesàro convergent to $A$, then for any $x\in X$ we know that $(d(x,A_n))_n$ is $\theta$-lacunary $f$-strong Cesàro convergent to $d(x,A)$. Therefore, according to Theorem in \cite{racsam}, $(d(x,A_n))_n$ is $\theta$-lacunary strong Cesàro convergent to $d(x,A)$ which yields the result.

On the other hand if $(A_n)_n$ is Wijsman lacunary strong Cesàro convergent to $A$ then for any $x\in X$, the sequence 
$(d(x,A_n))_n$ is $\theta$-lacunary strong Cesàro convergent to $d(x,A)$, but since $f$ is $\theta$-compatible we get that
$(d(x,A_n))_n$ is $\theta$-lacunary $f$-strong Cesàro convergent to $d(x,A)$ which is what we are looking for.
\end{proof}

\begin{theorem}
\label{reciprocolacunary}
\begin{description}
\item[a)] Assume that $WS_\theta^f=WS_\theta$ then $f$ is $\theta$-compatible.

\item[b)] Assume that $WN_\theta^f=WN_\theta$ then $f$ is $\theta$-compatible.
\end{description}
\end{theorem}
\begin{proof}
To show $a)$, let us assume that $f$ is not a $\theta$-compatible modulus function.
On the metric space $(\mathbb{R},|\cdot|)$ we will construct 
 a sequence $(B_n)_n\in WS_\theta\setminus WS_\theta^f$. An easy modification provides a counterexample on any metric space. 

Since $\varphi_{\theta}(\varepsilon)$ is an increasing function, if $f$ is not $\theta$-compatible then  there exists $c>0$ such that for any $\varepsilon>0$:  $\varphi_\theta(\varepsilon)>c$.

Let us fix $\varepsilon_k\to 0$. Thus, for each $k$ there exists $h_{r_k}$ such that $f(h_{r_k}\varepsilon_k)\geq cf(h_{r_k})$. Moreover, since $h_{r}$ is increasing we can suppose that 
\begin{equation}
    \label{lacudesigualdad}
    h_{r_k}(1-\varepsilon_k)-1>0.
\end{equation} 

Let us denote by $\lfloor x\rfloor$ the integer part of $x$.  We set $n_k=\lfloor h_{r_k}\varepsilon_k \rfloor+1$. According to equation (\ref{lacudesigualdad}), we get that $h_{r_k}-n_k>0$. Let us define the subset $A_k=[k_{r_k}-n_k,k_{r_k}]\cap \mathbb{N}\subset I_{r_k}$, and $A=\bigcup_kA_k$. 

Let us define
$$
B_k=\begin{cases}
\{1\} & k\in A\\
\{0\} & k\notin A
\end{cases}
$$
We claim that the sequence $(B_n)_n$, is Wijsman $\theta$-lacunary statistically convergent to $B=\{0\}$ but not Wijsman $\theta$-lacunary $f$-statistically convergent, a contradiction.

Let us observe that for any $x$, $d(x,B)=|x|$ and
$$
d(x,B_k)=\begin{cases}
|1-x| & k\in A\\
|x| & k\notin A.
\end{cases}
$$
Therefore, if $r\neq r_k$ for any $k$,   then
$$
 \frac{\#\{l\in I_{r}\,:\, |d(x,B_l)-d(x,B)|>\varepsilon\}}{h_{r}} \leq \frac{0}{h_{r}}=0.
$$
And for $r=r_k$:
$$
 \frac{\#\{l\in I_{r_k}\,:\, |d(x,B_l)-d(x,B)|>\varepsilon\}}{h_{r_k}}= \frac{n_k}{h_{r_k}}\to 0
$$
as $k\to \infty$. 
That is, $(B_k)_k$ is Wijsman $\theta$-lacunary statistical convergent to $B=\{0\}$.

On the other hand, 

$$
 \frac{f(\#\{l\in I_{r_k}\,:\, |d(x, B_l)-d(x,B)|>\varepsilon\})}{f(h_{r_k})}= \frac{f(n_k)}{f(h_{r_k})}\geq \frac{f(h_{r_k}\varepsilon_k)}{f(h_{r_k})}\geq c
$$
which yields the desired result.

For the part $b)$ an easy check show that the sequence $(B_k)$ constructed before belong to $WN_\theta\setminus WN_\theta^f$.
\end{proof}

Finally we wish to explore the connections between Wijsman lacunary $f$-statistical convergence and Wijsman lacunary $f$-strong Cesàro convergence. To this end, let us refine the the concept of sequence Wijsman Uniform integrable.

Let $\theta=(k_t)$ be a lacunary sequence. A sequence $(x_n)$ in a normed space $Y$ is said to be $\theta$-lacunary uniformly integrable if
$$
\lim_{M\to \infty} \sup_t \sum_{\substack{k\in I_t \\ \|x_k\|\geq M}}     \|x_k\|=0.
$$

\begin{definition}
Let $\theta=(k_t)$ be a lacunary sequence and $(X,d)$ a metric space. A sequence $(A_n)_n\subset CL(X)$ is said to be Wijsman $\theta$-lacunary uniformly integrable if for any $x\in X$ the sequence $(d(x,A_n))_n$ is  $\theta$-lacunary uniformly integrable. Let us denote by $WI_\theta$ the subsets of all Wijsman $\theta$-lacunary uniformly integrable sequences.
\end{definition}
Let us point out that if a sequence $(A_k)_k\subset CL(X)$ is Wijsman uniformly integrable then it is Wijsman $\theta$-lacunary uniformly integrable.

\begin{theorem}
\begin{description}
\item[a)]Assume that $\theta=(k_t)$ is a lacunary sequence and $f$ is a modulus function. Then $WN_\theta^f\subset WS_\theta^f$.

\item[b)] Assume that $\theta=(k_t)$ is a lacunary sequence and $f$ is a $\theta$-compatible modulus function. Then 
$WS_\theta^f\cap WI_\theta\subset WN_\theta^f$.
\end{description}
\end{theorem}
\begin{proof}
If $(A_k)_k\subset CL(X)$ is Wijsman $\theta$-lacunary $f$-statistical convergent to $A$ then for any $x\in X$,
$(d(x,A_n))_n$ is $\theta$-lacunary $f$-strong Cesàro convergent to $d(x,A)$ therefore according to Theorem 4.1 in \cite{racsam} we get that $(d(x,A_n))_n$ is $\theta$-lacunary $f$-statistical convergent to $d(x,A)$ which means that $(A_n)_n$ is Wijsman $\theta$-lacunary $f$-statistical convergent to $A$.

To prove $b)$, let $(A_k)\in WS_\theta^f\cap WI_\theta$, Wijsman $\theta$-lacunary $f$-statistical convergent to $A$.
Then for any $x\in X$ the sequence $(d(x,A_n))_n$ is $\theta$-lacunary $f$-statistical convergent to $d(x,A)$. Additionally $(d(x,A_n))_n$ is uniformly integrable. Therefore, since $f$ is $\theta$-compatible according to Theorem 4.3 in \cite{racsman} we get that $(d(x,A_n))_n$ is $\theta$-lacunary $f$-strong Cesàro convergent to $d(x,A)$, that is $(A_n)_n\in  WN_\theta^f$ as desired.
\end{proof}

Finally let us shown that $\theta$-compatibility is  again necessary for such phenomena.
\begin{theorem}
Let $\theta=(k_t)_t$ be a lacunary sequence.
Assume that for some modulus function $f$, $WS_\theta^f\cap WI_\theta\subset WN_\theta^f$ then $f$ is $\theta$-compatible.
\end{theorem}
\begin{proof}
Let us assume that $f$ is not a $\theta$-compatible modulus function.
On the metric space $(\mathbb{R},|\cdot|)$ we will construct 
 a sequence $(B_n)_n\in WS_\theta^f\cap WI_\theta\setminus WN_\theta^f$. An easy modification provides a counterexample on any metric space. 
 
 Since $f$ is not $\theta$-compatible given $(\varepsilon_k)$ a decreasing sequence converging to zero, there exists a subsequence $r_k$ such that $f(h_{r_k}\varepsilon_k)\geq cf(h_{r_k})$, for some $c>0$.
 
 Let us consider the following subsets
 $$
 B_l=\begin{cases}
\{\varepsilon_k\} & l\in (k_{r_k-1},k_{r_k}]\\
\{0\} & \textrm{other case}
 \end{cases}
 $$
 Since $(B_l)_l$ is bounded, $(B_l)_l\in WI_\theta$. Since $(\varepsilon_k)_k$ is decreasing $(B_l)_l\in WS_\theta^f$.
 
 However, if we take $x=0$ and $B=\{0\}$ we get that $B_k\overset{WN_\theta^f}{\nrightarrow } B$. Indeed
 
 $$
\frac{1}{f(h_{r_k})}f\left(\sum_{n\in I_{r_k}} |d(x,B_n)|\right)=\frac{f(h_{r_{k}}\varepsilon_k)}{f(h_{r_k})}\geq c
$$
 which gives the desired result.
\end{proof}
\section*{Conclusions}
The notion of statistical convergence modulated by an unbounded modulus function, allow us to define convergences of sets, more restrictive than the Wijsman statistical convergence defined through the classical statistical convergence. We conclude that in order to obtain new, more restrictive data filtering methods, it is necessary to consider non-compatible modulus functions, because for compatible modulus functions both convergence methods: Wijsman statistical convergence and Wijsman $f$-statistical convergence,  are equivalents. Such structure remain stable if we consider the lacunary version of these convergence methods.

\section*{Availability of data and material}
Not applicable.

\section*{Acknowledgements}
We want to thank to the Vicerretorado de Investigación of University of Cádiz who supported  partially this research.

\section*{Competing interests}
The author declare that they have no competing interests.

\section*{Funding}

 The author was supported by Ministerio de Ciencia, Innovaci\'on y Universidades under grant PGC2018-101514-B-I00.

\end{document}